\documentclass{fic-lmod}		

\newtheorem{theorem}{Theorem}[section]

\newtheorem{corollary}[theorem]{Corollary}	

\theoremstyle{definition}
\newtheorem{definition}[theorem]{Definition}

\theoremstyle{remark}

\numberwithin{equation}{section}

\newcommand{\GF}{\operatorname{GF}}
\newcommand{\period}{\operatorname{period}}
\newcommand{\Ztwo}{{\mathbb Z}_2}
\newcommand{\lcm}{\operatorname{lcm}}

\newcommand{\rpb}{{R. P. Brent}}
\newcommand{\realtilde}{\makebox{\tt\char'176}}	  
\newcommand{\lbrk}{{\linebreak[0]}}
\newcommand{\MC}{{Math.\ Comp.}}
\newcommand{\LNCS}{{Lecture Notes in Computer Science}}
\newcommand{\etal}{{\em{et~al.}}}
\newcommand{\ie}{{\em{i.e.}}}
\newcommand{\eg}{{\em{e.g.}}}
\newcommand{\smod}{{\,\bmod\,}}		
\newcommand{\mmod}{{\;\bmod\;}}		
\newcommand{\delhat}{{\widehat{\delta}}}
\newcommand{\dhat}{{\widehat{d}}}
\newcommand{\khat}{{\widehat{k}}}
\newcommand{\nhat}{{\widehat{n}}}
\newcommand{\shat}{{\widehat{s}}}

\begin{document}

\title[Almost Irreducible and Almost Primitive Trinomials]
	{Algorithms for Finding Almost Irreducible\\ 
	 and Almost Primitive Trinomials}

\author{Richard P.\ Brent}
\address{Oxford University Computing Laboratory,\\
Wolfson Building, Parks Road,\\
Oxford OX1~3QD, UK}

\author{Paul Zimmermann}
\address{LORIA/INRIA Lorraine\\
615 rue du jardin botanique\\
BP 101, 54602 Villers-l\`es-Nancy, France}


\date{17 April 2003}		

\dedicatory{\vspace*{-10pt}	
Dedicated to Hugh Cowie Williams on the occasion of his 60th birthday.}

\begin{abstract}
Consider polynomials over $\GF(2)$.
We describe efficient 
algorithms for finding trinomials with 
large irreducible (and possibly primitive) factors, and give examples
of trinomials 
having a primitive factor of degree~$r$ for all Mersenne exponents
$r = \pm 3 \mmod 8$ in the range $5 < r < 10^7$,
although there is no irreducible trinomial of degree~$r$.
We also give 
trinomials with a primitive factor of degree $r = 2^k$
for $3 \le k \le 12$.  These trinomials enable efficient 
representations of the finite field $\GF(2^r)$.
We show how trinomials with large primitive factors can 
be used efficiently in applications where primitive trinomials
would normally be used.
\vspace*{-10pt} 
\end{abstract}

\maketitle

\section{Introduction}
\label{sec:intro}

Irreducible and primitive polynomials over finite fields have many
applications in cryptography, coding theory, random number generation, etc.
See, for example, \cite{Gathen,Golomb,Knuth,Lidl94,Menezes}.

For simplicity we restrict our attention to the finite field
$\Ztwo = \GF(2)$; the generalization to other finite fields is straightforward. 
All polynomials are assumed to be in $\Ztwo[x]$,
and computations on polynomials are performed in $\Ztwo[x]$ or in 
a specified quotient ring.
A polynomial $P(x) \in \Ztwo[x]$ may be written as $P$
if the argument $x$ is clear from the context.
We recall some standard definitions.
\begin{definition}
\label{def:period}
A polynomial $P(x)$ with $P(0) \ne 0$ 
has {\em period} $\rho$ if $\rho$ is the least positive
integer such that $x^\rho = 1 \mmod P(x)$.
We say that $x$ has order $\rho \mmod P(x)$.
\end{definition}
\begin{definition}
\label{def:reducible}
A polynomial $P(x)$ is {\em reducible} if it has
nontrivial factors; otherwise it is \hbox{\em irreducible}.
\end{definition}
\begin{definition}
\label{def:primitive}
A polynomial $P(x)$ of degree~$n > 0$ 	  
is {\em primitive} if $P(x)$ is irreducible and has period $2^n-1$.
(Recall our assumption that $P(x) \in \Ztwo[x]$.)
\end{definition}
If $P(x)$ is primitive, then $x$ is a generator for the multiplicative
group of the field $\Ztwo[x]/P(x)$, giving a concrete representation
of $\GF(2^n)$. See Lidl and Niederreiter~\cite{Lidl94} or
Menezes~{\etal}~\cite{Menezes} for background information.

There is an interest in discovering primitive polynomials of high degree~$n$
for applications in random number generation~\cite{rpb132,rpb211}
and cryptography~\cite{Menezes}.
In such applications it is often desirable to use primitive polynomials with
a small number of nonzero terms, {\ie}~a small weight.
In particular, we are interested in
{\em trinomials} of the form $x^n + x^s + 1$, where $n > s > 0$
(so there are exactly three nonzero terms).

If $P(x)$ is irreducible and $\deg(P) = n > 1$,
then the order of $x$ in $\Ztwo[x]/P(x)$ is
a divisor of $2^n-1$. To test if $P(x)$ is primitive, we must test if
the order of $x$ is exactly $2^n-1$. To do this efficiently\footnote{%
Here ``efficiently'' means in time polynomial in the degree~$n$.}
it appears 
that we need to know the complete prime factorization of $2^n-1$. 
At the time of writing 
these factorizations are known for 
$n < 713$ 		
and certain larger~$n$, see~\cite{Cunningham}.

We say that $n$ is a {\em Mersenne exponent\ } if $2^n-1$ is
prime. In this case the factorization of $2^n-1$ is trivial and an
irreducible polynomial of degree~$n$ is necessarily primitive.
Large Mersenne exponents are known~\cite{GIMPS},
so there is a possibility
of finding primitive trinomials of high degree.
To test if a trinomial of 
prime 
degree~$n$ is reducible takes time $O(n^2)$,
so to test all trinomials of degree~$n$ takes time $O(n^3)$.

Several authors~\cite{Her92,Kumada,Kur91,Zie69b}
have computed primitive trinomials 
whose degree is a Mersenne exponent, up to some bound
imposed by the computing resources available.
Recently Brent, Larvala and Zimmermann~\cite{rpb199} gave a
new algorithm, more efficient than those used previously,
and computed all the primitive trinomials
of Mersenne exponent $n \le 3021377$ 
(subsequently extended to $n \le 6972593$).

For some $n \ge 2$, irreducible trinomials of given degree~$n$ do not exist.
Swan's theorem (see \S\ref{sec:Swan}) rules out $n = 0 \mmod 8$
and also most $n = \pm 3 \mmod 8$.  Since about half of the known
Mersenne exponents are $\pm 3 \mmod 8$, we can only hope to find
primitive trinomials of degree~$n$ for about half the Mersenne
exponents $n$.  

In the cases where primitive trinomials are ruled out by Swan's theorem, 
the conventional approach is to use primitive polynomials with more than three
nonzero terms.  A polynomial with an even number of nonzero terms is divisible 
by $x+1$, so we must use polynomials with five or more 
nonzero terms~\cite{Kumada,Kur91,Menezes}.  
This is considerably more expensive in applications because the number of
operations required for multiplication or division by a sparse polynomial
is approximately proportional to the number of nonzero terms.

In \S\ref{sec:Swan} we discuss Swan's theorem, then
in \S\ref{sec:almost} we introduce ``almost irreducible'' and
``almost primitive'' polynomials as a way of circumventing the implications
of Swan's theorem.  An algorithm (AIT) 
for computing almost irreducible trinomials,
and an extension (APT) for almost primitive trinomials,
are described in \S\ref{sec:algs}.  
Algorithm APT has been used to find almost primitive trinomials
with high degree in cases where Swan's theorem shows that primitive
trinomials do not exist.
Computational results and examples
are given in \S\S\ref{sec:computation}--\ref{sec:examples}.
In \S\ref{sec:Fermat} we give some computational results on almost 
primitive trinomials that 
are useful for representing the finite fields $\GF(2^{2^k})$, $k \le 12$.
In \S\ref{sec:implicit} we explain how to use almost 
irreducible/primitive trinomials efficiently in applications.
Finally, in \S\ref{sec:density}, 
we conclude with some theoretical results on the density of
almost irreducible and almost primitive polynomials, 
and some computational results on the 
density of the corresponding 
trinomials.
We thank Shuhong Gao and the referees 
for their comments on a draft of this paper.

\section{Swan's theorem and its implications}
\label{sec:Swan}

Swan's theorem
is a rediscovery of results of
Pellet~\cite{Pellet},
Stickelberger~\cite{Stickelberger}, Dickson~\cite{Dickson}
and Dalen~\cite{Dalen}~--
see Swan \cite[p.~1099]{Swan} 
and von zur Gathen~\cite{Gathen03}. 
Let $\nu(P)$ denote the number of irreducible factors 
(counted according to their multiplicity) of a polynomial~$P \in \Ztwo[x]$.
\begin{theorem}{\rm Swan~\cite[Corollary 5]{Swan}}.
\label{thm:swan}
Suppose $n > s > 0$, $n-s$ odd. 
Then $\nu(x^n + x^s + 1) = 0 \mmod 2$ 
iff one of the following holds:\\
a) $n$ even, 
   $\;n \ne 2s$, $\;ns/2 \mmod 4 \in \{0,1\}$;\\
b) 
   $2n \ne 0 \mmod s$, $\;n = \pm 3 \mmod 8$;\\
c) 
   $2n = 0 \mmod s$, $\;n = \pm 1 \mmod 8$.
\end{theorem}
If both $n$ and $s$ are odd, we can replace $s$ by $n-s$
(leaving the number of irreducible factors unchanged,
since $\nu(x^n + x^s + 1) = \nu(x^n + x^{n-s} + 1)$)
and apply Swan's theorem.  If $n$ and $s$ are both even,
then $T = x^n + x^s + 1$ 
is a square and $\nu(T)$ is even.
Thus, in all cases we can determine $\nu(T) \mmod 2$ using Swan's theorem.

Since a polynomial that has an even number of irreducible factors is
reducible, we have:

\begin{corollary}
\label{cor:swan}
If $n$ is prime, $n = \pm 3 \mmod 8$, $s \ne 2$, $s \ne n-2$,
then $x^n + x^s + 1$ is reducible over $\Ztwo$.
\end{corollary}

Corollary~\ref{cor:swan} shows that there are no
irreducible trinomials with degree a Mersenne exponent $n = \pm 3 \mmod 8$
(except possibly for $s = 2$ or $n-2$).
This appears to prevent us from
using trinomials with periods $2^n-1$ in these cases. 
Fortunately, there is a way to circumvent Swan's theorem and avoid
paying a significant speed penalty in most
applications of irreducible/primitive trinomials. We describe this
in the following section.

\section{Almost primitive trinomials}
\label{sec:almost}

Tromp, Zhang and Zhao~\cite{Tromp97} asked the following question:
given an integer $r>1$, do there exist integers $n, s$
such that \[G = \gcd(x^n + x^s + 1, x^{2^r-1} + 1)\] is a primitive
polynomial of degree~$r$? They verified that the answer is ``yes'' for
$r \le 171$, and conjectured that the answer is always ``yes''.

Blake, Gao and Lambert~\cite{Blake96} confirmed the conjecture for
$r \le 500$. They also relaxed the condition slightly and asked: 
do there exist integers $n, s$
such that $G$ has a primitive factor of degree~$r$?
Motivated by~\cite{Blake96}, we make some definitions.
\begin{definition}
\label{def:almost}
A polynomial $P(x)$ of degree~$n$ is
{\em almost primitive (almost irreducible)}
if $P(0) \ne 0$			
and $P(x)$ has a primitive (irreducible) factor 
of degree~$r$, for some $r > n/2$. 
We say that $P$ has {\em exponent} $r$
and {\em increment} $n-r$.
\end{definition}

For example, the
trinomial $x^{16} + x^3 + 1$ is almost primitive with exponent 13 and
increment 3, because
\[x^{16} + x^3 + 1 = (x^3 + x^2 + 1)D_{13}(x),\]
where
\[D_{13}(x) = x^{13} + x^{12} + x^{11} + x^9 + x^6 + x^5 + x^4 + x^2 + 1\]
is primitive. 
{From} a computational viewpoint it is more efficient 
to work in the ring $\Ztwo[x]/(x^{16} + x^3 + 1)$ than in
the field $F = \Ztwo[x]/D_{13}(x)$.
In \S\ref{sec:implicit} we outline how it is possible to work in the field $F$,
while performing most arithmetic in the ring $\Ztwo[x]/(x^{16} + x^3 + 1)$,
and without explicitly computing the dense primitive polynomial $D_{13}(x)$.

Note that, according to Definition~\ref{def:almost}, 
a primitive polynomial is {\em a~fortiori}
an almost primitive polynomial (the case $r = n$).
Similarly, an irreducible polynomial other than 
$1$ or $x$ is almost irreducible.
The restriction $r > n/2$ in Definition~\ref{def:almost} 
ensures that polynomials such as
$(x^3 + x + 1)^2$ are not regarded as almost irreducible.

In practice 
we choose the smallest possible increment~$\delta$ for given exponent~$r$,
{\eg}~in Tables~\ref{tab:almostprim}--\ref{tab:Fermat} we have $\delta \le 16$.
For most practical purposes, almost primitive trinomials of exponent $r$ 
and small increment are almost as useful as primitive trinomials 
of degree~$r$ (see~\S\ref{sec:implicit}).

\section{Searching for almost irreducible/primitive trinomials}
\label{sec:algs}

In this section we outline algorithms for finding almost irreducible
or almost primitive trinomials with large exponent~$r$.  
In the latter case we assume that the complete factorization of $2^r-1$
is known.  
The algorithms are generalizations of those given 
in~\cite{rpb199,Gathen,Menezes}, 
which handle the case $\delta = 0$.

\subsection{An algorithm for almost irreducible trinomials}
\label{subsec:irred}

Suppose $0 \le \delta < r$, $0 < s < r + \delta$,
and we wish to test if the trinomial $T(x) = x^{r+\delta} + x^s + 1$
is almost irreducible with exponent~$r$ (see Definition~\ref{def:almost}).
If it is not then we discard it, and (perhaps)
try again with different $(r, s,\delta)$.

We first state the algorithm, then explain the steps whose justification
is not immediately obvious.
Input to the algorithm is $(r, s, \delta)$ and a 
sieving bound $B \in [\delta, r)$. The optimal $B$ is
implementation-dependent: 
see the discussion in~\cite{rpb199}.
In the computation of Table~\ref{tab:almostprim} we used 
$B = \min(r-1,\max(\delta, 4 + \lfloor \log_2 r \rfloor))$.
Recall that polynomials are in $\Ztwo[x]$,
so computations on polynomials are performed in $\Ztwo[x]$ or in a 
quotient ring such as $\Ztwo[x]/T(x)$.\\

\hspace*{25pt}{\bf Algorithm AIT$(r,s,\delta,B)$}
\begin{enumerate}
\item\label{step1}			
If $\gcd(r+\delta,s) = 0 \mmod 2$ then return false.
\item\label{step2}
$d:=0$; $k := 0$; $S := 1$; $T := x^{r+\delta} + x^s + 1$;\\
for $i := 2$ to $\delta$ do\\ 
\hspace*{10pt} $g := \gcd(T, (x^{2^i} \mmod T) + x)$;\\
\hspace*{10pt} $g := g/\gcd(g,S)$; $S := g \times S$;\\
\hspace*{10pt} $d := d + \deg(g)$; $k := k + \deg(g)/i$;
\item\label{step3}
if $(d \ne \delta)$ or $(k = \nu(T) \mmod 2)$ 
then return false.
\item\label{step4}
for $i := \delta+1$ to $B$ do\\
\hspace*{10pt} $g := \gcd(T, (x^{2^i} \mmod T) + x)$;\\
\hspace*{10pt} if $S \mmod g \ne 0$ then return false.
\item\label{step5}
if $((x^{2^r} \mmod T) + x)S\ne 0 \mmod T$ then 
 return false.
\item\label{step6}
for each prime divisor $p \ne r$ of $r$\\
\hspace*{10pt} if $\gcd(((x^{2^{r/p}} \mmod T) + x)S, T) \ne S$
then return false.
\item\label{step7}
return true. [$T$ is almost irreducible with exponent $r$.]
\qed						
\end{enumerate}

If $\delta = 0$, Algorithm AIT
reduces to a standard algorithm for finding irreducible 
trinomials. 
We can assume that $\delta \ne 1$, because the only irreducible 
polynomials of degree~1 are $x$ and $x+1$, and neither can be a
factor of a trinomial. 
Hence, we need only consider 
$\delta \ge 2$.  

Step~\ref{step1} discards trinomials that are squares (see 
Theorem~\ref{thm:gcdirred} below). If this step is passed then 
$\gcd(T,T') = 1$, so $T$ is square-free.

Step~\ref{step2} may be called {\em sieving}, although it is done
by GCD computations.
By computing $\gcd(T, (x^{2^i} \mmod T) + x)$ 
for $2 \le i \le \delta$, we find the product
$S = P_1\cdots P_k$ 
of all irreducible factors $P_j$ of $T$ such that
$\deg(P_j) = d_j \le \delta$.  
We have $T = SD$, where $D$ is some polynomial of degree $n-d$,
and $\gcd(S,D) = 1$. 

Step~\ref{step3} returns false if $d$ or $k$ is incompatible with the
assumption that $T$ has an irreducible factor of degree $r$.
Note that Swan's theorem gives $\nu(T) \mmod 2$.

In Step~\ref{step4}, 
suppose $S \mmod g \ne 0$. Thus $f = g/\gcd(S,g)$ is a factor of $D = T/S$,
and $f \ne 1$.  If $f \ne D$ then $D$ is reducible. If $f = D$ then $D$
splits into factors of degree at most $B < r$, so again $D$ is reducible.
Thus, we can return false.

In Step~\ref{step5},
sieving has failed to discard $T$, so a full irreducibility test
of $D$ is required.
We can discard $T$ ({\ie}~return false) if $x^{2^r}\ne x \mmod D$,
but $D$ is in general a dense polynomial, so we perform an equivalent
computation that only involves exponentiation mod~$T$.
Note that the computation of $x^{2^r} \mmod T$ takes only 
$O(r^2)$ 			
bit-operations, since $T$ is a trinomial.

Finally, we should return false if
$\gcd(x^{2^q} + x, D) \ne 1$\	
for any divisor $q$ of $r$, $q \ne r$.
Step~\ref{step6} implements an equivalent test that is more
efficient because the operations are performed mod~$T$ and
only maximal divisors $q = r/p$ of~$r$ are checked.  
Step~\ref{step6} is trivial if~$r$ is prime.

\subsection{Algorithm APT for almost primitive trinomials}
\label{subsec:prim}

To search for\linebreak 
almost primitive trinomials with exponent~$r$,
we apply Theorem~\ref{thm:gcdprim} below,
and then algorithm AIT, 
to find a candidate trinomial that is almost irreducible with exponent~$r$. 
Unless $2^r-1$ is prime, it is necessary
to verify that the irreducible factor $D$
of degree~$r$ has period $2^r-1$ and not some proper divisor of $2^r-1$.
This can be done by verifying that,
for each prime divisor $p$ of $2^r-1$,
\[((x^{(2^r-1)/p} \mmod T) + 1)S \ne 0 \mmod T\;.\]

\subsection{Refinements}
\label{subsec:refinements}
Several refinements are possible.

{\bf 1.}
The fast algorithm of~\cite[\S4]{rpb199} 
can be used to accelerate the computation of
$x^{2^i} \mmod T$ in 
steps \ref{step2}, \ref{step4}--\ref{step6} of Algorithm~AIT 
if $r+\delta$ is odd. 

{\bf 2.}
Sieving can often be curtailed. Suppose that
step~\ref{step2} of Algorithm~AIT 
has been performed for $i \le \delhat < \delta$,
so we have found all $\khat$ irreducible factors of\linebreak 
degree $\le\delhat$.
Suppose that the sum of their degrees is $\dhat$. If
\begin{equation}
\dhat < \delta < \dhat + \delhat + 1,		\label{eq:testodd}
\end{equation}
then the constraint $d = \delta$ can not be satisfied
and we can return false.
Also, if $\khat \ne \nu(T) \mmod 2$, then~(\ref{eq:testodd}) 
can be replaced by the weaker condition
\begin{equation}
\dhat < \delta < \dhat + 2(\delhat + 1).	\label{eq:testeven}
\end{equation}

{\bf 3.}
If $r$ is a Mersenne exponent, computation of 
the small factor $S$ can be avoided.
Define $F = \lcm(2^{d_j} - 1)$, so $F$ is a multiple of the period
of $S$, and $\gcd(F, 2^r-1) = 1$. Step~\ref{step2} of Algorithm AIT
can easily be modified to compute~$F$ instead of~$S$,
and to save $g_i = \deg g$ at iteration~$i$.
Step~\ref{step4} can be modified to return false if
$\deg g > \sum_{j|i,2 \le j\le \delta} g_j$.
Step~\ref{step5} can be modified to return false if
\begin{equation}
(x^F)^{2^r} \ne x^F \mmod T(x).				\label{eq:easytest}
\end{equation}
The computation involved is almost the same as for the ``standard''
method of testing irreducibility of a trinomial~\cite[\S3]{rpb199}: 
the significant difference is that we start with $x^F$ instead of~$x$.
This variation of algorithm AIT was used to compute most of the
entries in Table~\ref{tab:almostprim}. 

{\bf 4.}
Theorems \ref{thm:gcdirred}--\ref{thm:gcdprim} 
allow us to discard 
many 
trinomials quickly.
\begin{theorem}
\label{thm:gcdirred}
Let $T(x) = x^n + x^s + 1$ be an almost irreducible trinomial.
Then $\gcd(n,s)$ is odd.
\end{theorem}
\begin{proof}
Assume that $\gcd(n,s)$ is even. Then 
$T(x) = (x^{n/2} + x^{s/2} + 1)^2$ is a square,
so can not have an irreducible factor of degree greater than $n/2$.
\end{proof}
We remark that $x^6 + x^3 + 1$ is irreducible (though not primitive)
over $\Ztwo$, and in this case $\gcd(n,s) = 3$.

\begin{theorem}
\label{thm:gcdprim}
Let $T(x) = x^n + x^s + 1$ be an almost primitive trinomial.
Then $\gcd(n,s) = 1$.
\end{theorem}
\begin{proof}
Suppose $g = \gcd(n,s) > 1$.  
{From} Theorem~\ref{thm:gcdirred} we 
can assume that $g \ge 3$. Write $y = x^g$.  
Thus $T(x) = y^{\nhat} + y^{\shat} + 1$, where $\nhat = n/g$, $\shat = s/g$.
The order of 
$y$ is at most $2^{\nhat}-1$.
Thus, the order of 
$x$ is at most $g(2^{\nhat}-1) = g(2^{n/g}-1)$.
If $T(x)$ is almost primitive with exponent $r$, then the order of
$x$ is $2^r-1$.  Thus
\[2^r-1 \le g(2^{n/g}-1).\]
Now $n+1 \le 2r$ by Definition~\ref{def:almost}, so
\begin{equation}
2^{(n+1)/2} - 1 \le g(2^{n/g}-1).			\label{eq:gnineq}
\end{equation}
The right-hand side of~(\ref{eq:gnineq}) is a decreasing function
of $g$ for $g \ge 3$. Thus,
\begin{equation}
2^{(n+1)/2} - 1 \le 3(2^{n/3} - 1).			\label{eq:ineq3} 
\end{equation}
It is easy to verify that~(\ref{eq:ineq3}) can not hold for $n \ge 6$,
but if $n < 6$ then $n/g < 2$, which is a contradiction.  Hence $g = 1$.
\end{proof}

\section{Computational results}
\label{sec:computation}

We conducted a search for
almost primitive trinomials whose exponent $r$ is also a Mersenne exponent.
For all Mersenne exponents $r = \pm 1 \mmod 8$ with $r \le 6972593$,
primitive trinomials of degree~$r$ are known, see~\cite{rpb199}.
Here we consider the cases $r = \pm 3 \mmod 8$, where the
existence of irreducible trinomials $x^r + x^s + 1$ 
is ruled out by
Swan's theorem (except for $s = 2$ or $r-2$, but the only known cases
are $r = 3, 5$).
For each exponent $r$ we searched for all almost primitive trinomials with the
minimal increment $\delta$ for which at least one almost primitive trinomial
exists. The search has been completed
for all Mersenne exponents $r < 10^7$.

In all cases of Mersenne exponent
$r = \pm 3 \mmod 8$, where $5 < r < 10^7$,
we have found at least one almost primitive trinomial
with exponent $r$ and some increment $\delta \in [2,12]$.
The results are summarized in Table~\ref{tab:almostprim}.
The first four entries are from~\cite[Table~4]{Blake96}; 
the other entries are new.
They were computed using 
Algorithm APT, with some simplifications that are possible
because $r$ is a Mersenne exponent (see~\S\ref{subsec:refinements}.3 above).

For all but two of the almost primitive trinomials
$x^{r+\delta}+x^s+1$ given in Table~\ref{tab:almostprim}, 
the period $\rho = (2^r-1)f$ satisfies
$\rho > 2^{r+\delta-1}$.
Thus, the period is greater than the greatest period ($2^{r+\delta-1}-1$)
that can be obtained for any polynomial
of degree less than $r + \delta$.
In the two exceptional cases
the small factors of degree~$8$ are
not primitive,  having period $85 = 255/3$.
 
For the largest known Mersenne exponent, $r = 13466917$,
we have not yet started an extensive computation, but 
we have shown that $\delta \ge 3$ and $\delta \ne 4$
(see Theorem~\ref{t13466917}).

\begin{table}
\begin{center}
\begin{tabular}{|c|c|c|c|c|} \hline
$r$  &$\delta$	&$s$	&$f$	&Small factors and remarks\\ 
\hline
13	&3	&3	&7	&$x^3 + x^2 + 1$\\
\hline
19	&3	&3	&7	&$x^3 + x + 1$\\
\hline
61	&5	&17	&31	&$x^5 + x^3 + x^2 + x + 1$\\ 
\hline
107	&2	&8	&3	&$x^2 + x + 1$\\
	&	&14	&3	&$x^2 + x + 1$\\
	&	&17	&3	&$x^2 + x + 1$\\
\hline
2203	&3	&355	&7	&$x^3 + x^2 + 1$\\
\hline
4253	&8	&1806	&255	&$x^8 + x^7 + x^2 + x + 1$\\
	&	&1960	&85	&$x^8 + x^6 + x^5 + x^4 + x^2 + x + 1$\\
\hline
9941	&3	&1077	&7	&$x^3 + x^2 + 1$\\
\hline
11213	&6	&227	&63	&$x^6 + x^5 + x^3 + x^2 + 1$\\
\hline
21701	&3	&6999	&7	&$x^3 + x^2 + 1$\\
	&	&7587	&7	&$x^3 + x^2 + 1$\\
\hline
86243	&2	&2288	&3	&$x^2 + x + 1$\\
\hline
216091	&12	&42930	&3937	&$x^{12} + x^{11} + x^5 + x^3 + 1$\\
	&	&	&	&$=(x^5 + x^4 + x^3 + x + 1)\cdot$\\
	&	&	&	&$(x^7 + x^5 + x^4 + x^3 + x^2 + x + 1)$\\
\hline
1257787	&3	&74343	&7	&$x^3 + x^2 + 1$\\ 
\hline 
1398269	&5	&417719	&21	&$x^5 + x^4 + 1 = (x^2+x+1)\cdot(x^3+x+1)$\\
\hline
2976221	&8	&1193004 &85	&$x^8 + x^7 + x^6 + x^5 + x^4 + x^3 + 1$\\
\hline
13466917&?	&?	&?	&None for $\delta < 3$ or $\delta = 4$\\
\hline
\end{tabular}
\vspace*{10pt}
\caption{}
\vspace*{-10pt}
{\small
\begin{center}
Some almost primitive trinomials over $\Ztwo$.\\
$x^{r+\delta}+x^s+1$ has a primitive factor of degree~$r$;\\
$\delta$ is minimal; $2s \le r+\delta$; the period $\rho=(2^r-1)f$.
\end{center}
} 
\label{tab:almostprim}
\end{center}
\end{table}

\section{Examples and special cases}
\label{sec:examples}

We considered the almost primitive trinomial $x^{16} + x^3 + 1$ in
\S\ref{sec:almost}. Here we give an example with much higher degree:
$r = 216091$, $\delta = 12$. We have
\[x^{216103} + x^{42930} + 1 = S(x)D(x),\]
where \[S(x) = x^{12} + x^{11} + x^5 + x^3 + 1,\]
and $D(x)$ is a (dense) primitive polynomial of degree 216091.
The factor $S(x)$ of degree 12 splits into a product of two primitive
polynomials, $x^5 + x^4 + x^3 + x + 1$ and 
$x^7 + x^5 + x^4 + x^3 + x^2 + x + 1$. 
The contribution to the
period from these factors is $f = \lcm(2^5-1,2^7-1) = 3937$.

\begin{theorem}
\label{t216091}
If $T$ is an almost irreducible trinomial with exponent 216091
and increment $\delta$, then $\delta \notin \{0,1,2,4,6\}$.
\end{theorem}
\begin{proof}
As $T$ is a trinomial, it is not divisible by $x$ or $x+1$, so
$\delta \ne 1$.\linebreak
Assume that $T = x^n + x^s + 1$ is almost irreducible 
with $\deg(T) = n = r + \delta$, where\linebreak 
$\delta \in \{0,2,4,6\}$.
Thus $T = SD$ where $D$ is irreducible, $\deg(D) = r = 216091$,
$n$ is odd and we can assume that $s$ is even (otherwise
replace $s$ by $n-s$).
Since $r$ is a Mersenne exponent, $D$ is primitive and $T$ is
almost primitive.
Let $\nu_2 = \nu(T) \mmod 2$. We consider the cases $\delta = 0, 2, 4, 6$
in turn and show that each case leads to a contradiction.\\
{\bf A. $\delta = 0$.} $n = 3 \mmod 8$ and $\nu_2 = 1$, so by Swan's theorem
$s \;|\; 2n$. {From} Theorem~\ref{thm:gcdprim}, $\gcd(n,s)=1$, so the only
possibility is $s=2$.
Now $n = 1 \mmod 3$, so $x^n + x^2 + 1 = x^2 + x + 1 \mmod x^3 - 1$.
Thus $x^2 + x + 1 \;|\; x^n + x^2 + 1$, a contradiction.\\
{\bf B. $\delta = 2$.} 
$n = 0 \mmod 3$, so $T = x^s \mmod x^3 - 1$.
Thus $x^2 + x + 1$ does not divide $T$,	
but $x^2 + x + 1$ is the only irreducible
polynomial of degree~2, and hence the only possibility for $S$, 
a contradiction.\\
{\bf C. $\delta = 4$.}
$\deg(S) = 4$, but $S$ can not have irreducible factors of degree~2,
since $x^2 + x + 1$ is the only irreducible polynomial of degree~2, 
and $T$ is square-free.  Thus $S$
is irreducible of degree~4 and a divisor of $x^{15}-1$.  
We have $n = -1 \mmod 8$ 
and $\nu_2 = 0$, so by Swan's theorem 
and Theorem~\ref{thm:gcdprim} 
we must have $s = 2$.
Now $n = 5 \mmod 15$, so
$T = x^5 + x^2 + 1 \mmod x^{15} - 1$, 
but $x^5 + x^2 + 1$ is irreducible.
Thus $T$ has no factor of degree~4, a contradiction.\\
{\bf D. $\delta = 6$.}
$S$ could be of the
form $S_6$, $S_2S_4$ or $S_3\widehat{S}_3$,
where $S_j, \widehat{S}_j$ are irreducible of degree~$j$.
If $S = S_6$ then $\nu_2 = 0$ and, by Swan's theorem and
Theorem~\ref{thm:gcdprim}, $s = 2$.  However, $n = 1 \mmod 3$,
so $x^2 + x + 1\;|\;x^n + x^2 + 1$, a contradiction.
If $S = S_2S_4$ then $\deg(\gcd(T,x^{15}-1)) = 6$.
Now $n = 7 \mmod 15$, so $T = x^7 + x^{s \smod 15} + 1 \mmod x^{15}-1$,
and in each of the 15 cases we find that $\deg(\gcd(T,x^{15}-1)) \le 4$.
If $S = S_3\widehat{S}_3$ then 
$\deg(\gcd(T,x^7-1)) = 6$.
Now\linebreak $n = 0 \mmod 7$, so $T = x^{s \smod 7} \mmod x^7-1$, and
$\gcd(T,x^7-1) = 1$. Thus $\delta \ne 6$.
\end{proof}
\begin{theorem}
\label{t2976221}
If $T$ is an almost irreducible trinomial with exponent 2976221
and increment $\delta$, then $\delta \notin \{0,1,2,4\}$.
\end{theorem}
\begin{proof}
The proof is similar to that of Theorem~\ref{t216091}.
Assume that $\delta \in \{0,2,4\}$ and that $s$ is even.
If $\delta = 0$ we must have $s = 2$.
Now $n = 3 \mod 7$, so\linebreak
$T = x^3 + x^2 + 1 \mmod x^7 - 1$, 
and thus $T$ has an irreducible factor $x^3 + x^2 + 1$.
If $\delta = 2$, again we must have $s = 2$. 
In this case $T = x^{118} + x^2 + 1 \mmod x^{255}-1$,
and a computation shows that 
$x^8 + x^7 + x^3 + x^2 + 1\;|\;T$.
Finally, if $\delta = 4$, we have
$T = x^{s \smod 15} \mmod x^{15}-1$,
so $T$ has no irreducible factor of degree~4.
\end{proof}
\begin{theorem}
\label{t13466917}
If $T$ is an almost irreducible trinomial with exponent 
13466917 and increment $\delta$, then $\delta \notin \{0,1,2,4\}$.
\end{theorem}
\begin{proof}
As above, we can assume that $\delta \ne 1$, $n = r + \delta$ is odd,
and without loss of generality $s$ is even.
If $\delta = 0$ the only case to consider is $s = 2$, but 
$n = 1 \mmod 3$, so $T = x^2 + x + 1 \mmod x^3 - 1$, and thus
$T$ is divisible by $x^2 + x + 1$.
If $\delta = 2$ then $T = x^{s \smod 3} \mmod x^3 - 1$, so $T$ is never
divisible by $x^2 + x + 1$.
If $\delta = 4$ then $S$ must be irreducible of degree~4, and 
the only possibility is $s=2$.
Now $T = x^{11} + x^2 + 1 \mmod x^{15} - 1$, but $x^{11} + x^2 + 1$
is irreducible, so $T$ has no irreducible factor of degree~4.
\end{proof}

\section{The Fermat connection}
\label{sec:Fermat}

If we have found an almost irreducible trinomial
$T = x^n + x^s + 1$ with exponent $r = n - \delta$,
then to check if $T$ is almost primitive we need the
complete factorization of $2^r-1$.  
In~\S\ref{sec:computation} we chose $r$ so the factorization was
trivial, because $2^r-1$ was a Mersenne prime. Another case of interest,
considered in this section, is when $r$ is a power of two, say $r = 2^k$. 
Then
\[2^r - 1 = F_0 F_1 \cdots F_{k-1},\]
where $F_j = 2^{2^j}+1$ is the $j$-th Fermat number.
The complete factorizations of these $F_j$ are known for $j \le 11$
(see~\cite{rpb161}) so we can factor $2^{2^k}-1$ for $k \le 12$. 

In Table~\ref{tab:Fermat} we give almost primitive trinomials 
$T = x^{r+\delta}+x^s+1$ with exponent
$r = 2^k$ for $3 \le k \le 12$.
Thus $T = SD$ where $D$ is primitive and has degree~$2^k$.
We also give $S$ in factored form.
The irreducible factors of $S$ are not always primitive.
The period of $T$ is $\lcm(2^r-1,\period(S)) = (2^r-1)f$.

By Swan's theorem,
a primitive trinomial of degree $2^k$ does not exist for $k \ge 3$.
However, we can work efficiently in the finite fields $\GF(2^{2^k})$,
$k \in [3, 12]$, using the trinomials listed in Table~\ref{tab:Fermat} 
and the implicit algorithms of~\S\ref{sec:implicit}.

\begin{table}
\begin{center}
\begin{tabular}{|c|c|c|c|c|c|} \hline
$k$ 	&$r$  &$\delta$	&$s$	&$f$	&Small factor $S(x)$\\ 
\hline
3	&8 	&5 	&1 	&31	&$x^5 + x^4 + x^3 + x + 1$\\
	&	&	&2	&7	&$(x^2 + x + 1)(x^3 + x + 1)$\\ 
\hline
4	&16	&11	&2	&7	&$(x^3 + x + 1)
				  (x^8 + x^7 + x^6 + x^3 + x^2 + x + 1)$\\
\hline
5	&32	&8	&3	&1	
	&$x^8 + x^6 + x^5 + x^4 + x^2 + x + 1$\\
\hline
6	&64	&10	&3	&21	
	&$(x^4 + x + 1)(x^6 + x^5 + x^4 + x + 1)$\\
	&	&	&21	&341	
	&$x^{10} + x^7 + x^6 + x^5 + x^3 + x^2 + 1$\\
\hline
7	&128	&2	&17	&1	&$x^2 + x + 1$\\
\hline
8	&256	&16	&45	&1	&$x^{16} + x^{15} + x^{14} + x^{11} + x^9 + 
				   x^7 + x^3 + x + 1$\\
\hline
9	&512	&9	&252	&31	&$(x^4 + x + 1)(x^5 + x^3 + 1)$\\
\hline
10	&1024	&3	&22	&7	&$x^3 + x^2 + 1$\\
\hline
11	&2048	&10	&101	&341	
	&$x^{10} + x^9 + x^8 + x^7 + x^5 + x^4 + 1$\\
\hline
12	&4096	&3	&600	&7	&$x^3 + x + 1$\\ 
	&	&	&628	&7	&$x^3 + x + 1$\\
	&	&	&1399	&7	&$x^3 + x^2 + 1$\\
\hline
\end{tabular}
\vspace*{10pt}
\caption{}
\vspace*{-10pt}
{\small
\begin{center}
Some almost primitive trinomials over $\Ztwo$.\\
$x^{r+\delta}+x^s+1$ has a primitive factor of degree~$r = 2^k$;\\
$\delta$ is minimal; $2s \le r+\delta$; the period $\rho=(2^r-1)f$.
\end{center}
} 
\label{tab:Fermat}
\end{center}
\end{table}

\section{Implicit algorithms}
\label{sec:implicit}

Suppose we wish to work in the finite field $\GF(2^r)$ where $r$ is the
exponent of an almost primitive trinomial $T$.
We can write
$T = SD$,
where $\deg(S) = \delta$, $\deg(D) = r$.
Thus
\[GF(2^r) \equiv \Ztwo[x]/D(x),\]
but because $D$ is dense we wish to avoid working directly with $D$,
or even explicitly computing $D$.  We show that it is possible to
work modulo the trinomial $T$.  

We can regard $\Ztwo[x]/T(x)$ as a redundant representation of
$\Ztwo[x]/D(x)$.
Each element $A \in \Ztwo[x]/T(x)$ can be represented
as 
\[A = A_c + A_dD,\]
where $A_c \in \Ztwo[x]/D(x)$ is the ``canonical representation'' 
that would be obtained if we worked in $\Ztwo[x]/D(x)$,
and $A_d \in \Ztwo[x]$ is some
polynomial of degree less than~$\delta$.

We can perform computations such as addition, multiplication
and exponentiation in $\Ztwo[x]/T(x)$, taking advantage of the sparsity
of $T$ in the usual way.\linebreak  
If $A \in \Ztwo[x]/T(x)$ and we wish to
map $A$ to its canonical representation $A_c$,
we use the identity
\[A_c = (AS \mmod T)/S,\]
where the division by the (small) polynomial $S$ is exact.
A straightforward implementation requires only $O(\delta r)$ operations.
We avoid computing $A_c = A \mmod D$ directly; in fact we never
compute the (large and dense) polynomial $D$: 
it is sufficient that $D$ is determined
by the trinomial $T$ and the small polynomial~$S$.

In applications such as random number generation~\cite{rpb211},
where the trinomial
$T = x^n + x^s + 1$ is the denominator of the
generating function for a linear recurrence $u_k = u_{k-s} + u_{k-n}$,
it is possible (by choosing appropriate initial conditions that annihilate
the unwanted component) to generate
a sequence that satisfies the recurrence defined by the 
polynomial $D$. 
However, this is not necessary if all that matters is that the
linear recurrence generates a sequence with period at least $2^r-1$.

\section{The density of almost irreducible/primitive polynomials}
\label{sec:density}

In this section we state some theorems regarding the distribution of
almost irreducible {\em polynomials}.  The proofs are straightforward,
and similar to the proof of Theorem~1.2 in~\cite{Gao99},
which generalizes our Corollary~\ref{cor:density1}.
We would like to prove similar theorems about almost irreducible
{\em trinomials}, but this seems to be difficult.

Let $I_n$ denote the number of irreducible polynomials of degree~$n$, 
excluding the polynomial $x$, and let
$N_{r,\delta}$ denote the number of almost irreducible polynomials with
exponent $r$ and increment $\delta$.  
Thus
$\sum_{d|n} d I_d = 2^n - 1$
and, by M\"obius inversion,
\[I_n = \frac{1}{n} \sum_{d|n} \mu(d)\left(2^{n/d}-1\right)
      = \frac{2^n}{n}\left(1 + O(2^{-n/2})\right)\;.\]	
\begin{theorem}
\label{thm:density}
If $0 \le \delta < r$, then $N_{r,\delta} = 2^{\max(0,\delta-1)}I_r$.
\end{theorem}
\begin{proof}
The case $\delta=0$ is immediate, so assume that $\delta > 0$.
Thus $r > 1$.
For each irreducible polynomial $D$ of degree $r$,
and each polynomial $S$ of degree $\delta$ such that $S(0) \ne 0$,
there is an almost irreducible polynomial $P = SD$. Also, by
the constraint $\delta < r$, $P$ determines $S$ and $D$ uniquely.
Thus, the result follows by a counting argument, since there are
$2^{\delta-1}$ possibilities for~$S$ and $I_r$ possibilities for~$D$.
\end{proof}
\begin{corollary}	
\label{cor:density1}
If $n = r + \delta$, $0 < \delta < r$, and
$P$ is chosen uniformly at random from the $2^{n-1}$ polynomials of 
degree $n$ with $P(0) \ne 0$, 
then the probability that $P$ is almost irreducible
with exponent $r$ is $\frac{1}{r}\left(1 + O(2^{-r/2})\right)$.
\end{corollary}
\begin{corollary}
\label{cor:density2}
If $n \ge 1$ and
$P$ is chosen uniformly at random from the $2^{n-1}$ polynomials of 
degree $n$ with $P(0) \ne 0$, 
then the probability that $P$ is almost\linebreak 
irreducible with some valid exponent 
is $\ln 2 + O(1/n)$.
\end{corollary}

\subsection*{An analogy}
There is an analogy between polynomials of degree $n$ and\linebreak
$n$-digit numbers,
with irreducible polynomials corresponding to primes. 
A result analogous to Corollary~\ref{cor:density2} is:
the probability that a random $n$-digit number has a prime factor
exceeding $n/2$ digits is $\ln 2 + O(1/n)$ 
(see for example~\cite{Norton}).

\subsection*{The density of almost primitive polynomials}
The number of primitive polynomials of degree $n$ is
$P_n = \phi(2^n-1)/n$, where $\phi$ denotes Euler's phi function,
see for example 
Lidl~\cite{Lidl94}. 
If $I_r$ is replaced by $P_r$ in Theorem~\ref{thm:density}, then we
obtain the number of almost primitive polynomials with exponent~$r$
and increment~$\delta$. It is easy to deduce an analogue of 
Corollary~\ref{cor:density1}.
To obtain an analogue of Corollary~\ref{cor:density2} for almost primitive
polynomials we would need to estimate 
$\sum_{n/2 < r \le n}\phi(2^r-1)/(r2^r)$.
For $n = 1000$ the approximate value is 0.507.

\subsection*{Computational results for trinomials}

Let $N_{ait}(n)$ be the number of almost irreducible trinomials
$x^n + x^s + 1$ with $0 < s < n$. Consider the smoothed and
normalized value
$E_{ait}(n) = \frac{2}{n(n-1)}\sum_{m=2}^{n} N_{ait}(m)$. 
If a result like Corollary~\ref{cor:density2}
applies for {\em trinomials}, at least in the limit
as $r$, $\delta \to \infty$, then it is plausible to conjecture that
\begin{equation}
\lim_{n \to \infty}E_{ait}(n) = c 			      \label{eq:conj1}
\end{equation}
for some positive constant $c$.
We have computed $N_{ait}(n)$ and $E_{ait}(n)$ for $n \le 1000$;\linebreak
the numerical results support the conjecture (\ref{eq:conj1}) with
$c < \ln 2 \approx 0.6931$.
For example, $E_{ait}(500) \approx 0.4765$ and $E_{ait}(1000) \approx 0.4713$.
For almost primitive trinomials the corresponding limit seems
smaller (if it exists). Our computations give 
$E_{apt}(500) \approx 0.3124$ and $E_{apt}(1000) \approx 0.3104$.

\subsection*{Existence of almost irreducible/primitive trinomials}

We have shown by computation that an almost 
irreducible trinomial of
degree~$n$ exists for all\linebreak
$n \in [2,10000]$.  
Similarly, we have shown that an almost primitive trinomial of
degree~$n$ exists for all
$n \in [2,2000]\backslash\{12\}$.
In the exceptional case (degree~12),
$x^{12} + x + 1$ has primitive factors of degrees 3, 4, and~5, 
but degree~$5$ is too small,
so $x^{12} + x + 1$ is not ``almost primitive'' by Definition~\ref{def:almost}. 
The other candidate that is not excluded by Theorem~\ref{thm:gcdprim} is
$x^{12} + x^5 + 1$; this is irreducible but not primitive, 
having period $(2^{12}-1)/5$.

Rather than asking for an almost irreducible (or almost primitive) trinomial
of given degree, we can ask for one of given exponent.
This is close to the spirit of~\cite{Blake96,Tromp97}.
For all $r \in [2,2000]$ there is an almost irreducible trinomial
$x^{r+\delta} + x^s + 1$
with exponent $r$ and (minimal) increment $\delta = \delta_{ait}(r) \le 23$.
The extreme increment $\delta_{ait}(r) = 23$ 
occurs for $(r,s) = (1930,529)$,
and the mean value of $\delta_{ait}(r)$ 
for $r \in (1000,2000]$ is $\approx 2.14$.
A plausible conjecture is that $\delta_{ait}(r) = O(\log r)$.

Similarly, for all $r \in [2,712]$	
there is an almost primitive trinomial 	
with exponent $r$ and (minimal) increment $\delta_{apt}(r) \le 43$.
The extreme $\delta_{apt}(r) = 43$ 
occurs for $(r,s) = (544,47)$,
and the mean value of $\delta_{apt}(r)$ 
for $r \in (356,712]$ is $\approx 3.41$.

\pagebreak[3]

\end{document}